\documentclass[a4paper]{article}
\usepackage{amsthm,amsmath,amssymb}
\usepackage{enumerate,enumitem}
\usepackage{graphicx}
\usepackage{xfrac}

\usepackage{tabularx,booktabs,multirow}
\newcolumntype{C}{>{\centering\arraybackslash}X}
\newcolumntype{D}{>{\centering\arraybackslash}X}

\DeclareGraphicsExtensions{.pdf,.png,.eps}
\graphicspath{{figs/}}

\newtheorem{theorem}{Theorem}
\newtheorem{lemma}[theorem]{Lemma}
\newtheorem{observation}[theorem]{Observation}

\newtheorem{proposition}[theorem]{Proposition}
\newtheorem{corollary}[theorem]{Corollary}

\newtheorem*{claim*}{Claim}

\theoremstyle{remark}

\newcommand{\B}{\ensuremath{\mathcal{B}}}

\newcommand{\cF}{\ensuremath{\mathcal{F}}}

\usepackage[usenames,dvipsnames]{xcolor}
\begin{document}

\title{The Tur\'an number of Berge-$K_4$ in triple systems}

\author{ Andr\'as Gy\'arf\'as\thanks{Renyi  Institute of Mathematics, Budapest, Hungary}\thanks{Research supported in part by
NKFIH Grant No. K116769.}}

\maketitle

\begin{abstract}
A Berge-$K_4$ in a triple system is a configuration with four vertices $v_1,v_2,v_3,v_4$ and six distinct triples $\{e_{ij}: 1\le i< j \le 4\}$ such that $\{v_i,v_j\}\subset e_{ij}$ for every $1\le i<j\le 4$. We denote by $\B$ the set of Berge-$K_4$ configurations. A triple system is $\B$-free if it does not contain any member of $\B$. We prove that the maximum number of triples in a $\B$-free triple system on $n\ge 6$ points is obtained by the balanced complete $3$-partite triple system: all triples $\{abc: a\in A, b\in B, c\in C\}$ where $A,B,C$ is a partition of $n$ points with
$$\left\lfloor{n\over 3}\right\rfloor=|A|\le |B|\le |C|=\left\lceil{n\over 3}\right\rceil.$$
\end{abstract}

\section{Introduction}

Extending the notion of Berge paths and cycles, Gerbner and Palmer \cite{GP} introduced the notion of Berge-$G$ hypergraphs as follows.
A hypergraph $H=(V,F)$ is called Berge-$G$ if $G=(V,E)$ and there exists a bijection $g: E(G)\mapsto E(H)$ such that for $e\in E(G)$ we have $e\subset g(e)$.  For a fixed $G$, let $\B(G)$ denote the family of Berge-$G$ hypergraphs. For a family $\cF$ of $r$-uniform hypergraphs,  the Tur\'an number ${\rm ex}_r(n, \cF)$ is the largest number of edges in an $r$-uniform hypergraph on $n$ vertices that does not contain any member from a family $\cF$ as a subhypergraph.

Many asymptotic results are known for ${\rm ex}_r(n, \B(G))$, Gy\H ori and G. Y. Katona \cite{GK}, Gy\H ori, Lemons \cite{GL}, Gerbner and Palmer \cite{GP}, F\"uredi and \"Ozkahya \cite{FO},  Gerbner, Methuku and Vizer \cite{GMV}, Palmer, Tait, Timmons and Wagner \cite{PTTW}, Gr\'osz, Methuku, Tompkins. \cite{GMT}. Results of Alon and Shikhelman \cite{AS}, Gerbner and Palmer \cite{GP2} also imply bounds for ${\rm ex}_r(n, \B(G))$.

In case of $G=K_3$ a spectacular sharp result is that ${\rm ex}_3(n,\B(K_3))=\lfloor{n^2\over 8}\rfloor$ for $n\ge 3$. This was proved (among other results) independently by Gy\H ori \cite{G} and by Frankl, F\"uredi and Simonyi \cite{FFS} with an elegant reduction to Mantel's theorem. For $t\ge 13$, Maherani and Shahsiah \cite{MS} determined exactly ${\rm ex}_3(n,\B(K_t))$.

Here we focus on ${\rm ex}_3(n,\B(K_4))$ and from now on  $\B=\B(K_4)$. We refer to $3$-uniform hypergraphs as {\em triple systems}. A (special case of a) result of Mubayi \cite{M} is that ${\rm ex}_3(n, \B)$ is asymptotic to the maximum number of triples in a  $3$-partite triple system on $n$ vertices, i.e.
$${\rm ex}_3(n, \B)\sim f(n)=\left\lfloor{n\over 3}\right\rfloor \left\lfloor{n+1\over 3}\right\rfloor \left\lfloor{n+2\over 3}\right\rfloor.$$

Mubayi's result was sharpened by Pikhurko \cite{P} who proved that ${\rm ex}_3(n,\B)=f(n)$, for large enough $n$. In fact, both results were proved in a stronger form, only {\em one} specific member of $\B$  was forbidden,  the expansion, where the bijection $g$ maps each edge $e\in K_4$ to a triple $e\cup v_e$ so that the vertices $v_e$ are all different and different from the vertices of $K_4$ as well. Thus the expansion of $K_4$ has six triples on ten vertices, it is the largest member of $\B$.

We prove that  ${\rm ex}_3(n, \B)$ is equal to $f(n)$ for {\em every $n\ge 6$}.  For the small cases, ${\rm ex}_3(3,\B)=1$, ${\rm ex}_3(4, \B)=4$ are obvious, ${\rm ex}_3(5,\B)=5$ will be proved in Theorem \ref{exbk4n=56}.

\begin{theorem}\label{exbk4} For $n\ge 6$, ${\rm ex}_3(n, \B)=f(n)$.

\end{theorem}

The proof of Theorem \ref{exbk4} does not use that all configurations of $\B$  are excluded.  In fact, the exclusion of the ``maximal'' member, the extension of $K_4$, is not used. This gives the following corollary.

\begin{corollary}\label{cor} Let $B^*$ denote the expansion of $K_4$. Then  ${\rm ex}_3(n,\B\setminus B^*)=f(n)$ for $n\ge 6$.
\end{corollary}

Note that Pikhurko's result and Corollary \ref{cor} shows that there is no {\em unique} minimal subset $\B'\subset \B$ for which ${\rm ex}_3(n,\B')=f(n)$ for large enough $n$.

Assume that $H$ is a triple system and $A$ is a proper nonempty subset of $V(H)$.  The {\em trace of $H$} is a  multigraph $G_H(A)$ (multiloops allowed) on vertex set $V(H)\setminus A$. Every triple $xyz$ of $H$ with $x,y\in A, z\in V(H)\setminus A$ defines a loop on $z$ with label $\{x,y\}$ and every triple $xyz$ of $H$ with $x\in A, y,z\in V(H)\setminus A$ defines an edge $yz$ with label $\{x\}$. If $H$ is clear from the context, we use simply $G(A)$ instead of $G_H(A)$. The label of a loop or an edge is denoted by $\ell(v,v), \ell(v,w)$, respectively. Let $\mu(v), \mu(e)$ denote the multiplicity of a loop $vv$ and an edge $e$, respectively. When $|A|=1$, the trace is a simple graph, usually called the link of $A$. In a triple system the {\em degree of a vertex} $v$ is the number of triples containing $v$ and denoted by $d(v)$.

Theorem \ref{exbk4} is proved by combining induction on $n$ with  ${\rm ex}(m,K_4)=\lfloor {m^2\over 3}\rfloor$ (Tur\'an theorem) applied to a subgraph of $G=G(A)$ where $A$ is a triple in a suitable Berge-triangle of a $\B$-free $H_n$.  I have been recently informed by D. Gerbner that a different approach \cite{GMP} also implies Theorem \ref{exbk4} for $n\ge 9$.

\section{Proof of Theorem \ref{exbk4} and its corollary}

\begin{proof}

We prove Theorem \ref{exbk4} by induction. The base case $n=6$ (together with ${\rm ex}_3(5,\B)=5$) is proved in Theorem \ref{exbk4n=56}.  Assuming the theorem is true for $n-1\ge 6$, let $H_n$ be a $\B$-free triple system with $n$ vertices and $f(n)+1$ edges. From this we will get a contradiction.

We can apply induction if for some vertex $v$ of $H_n$, $d(v)\le f(n)-f(n-1)$. Indeed, then deleting $v$ and the triples containing $v$ we get a $B(K_4)$-free triple system
$H'$ with $n-1$ vertices and more than $f(n-1)$ edges, contradiction. Thus we may assume that the minimum degree $\delta=\delta(H_n)$ is larger than $f(n)-f(n-1)$. The actual values of $f(n)-f(n-1)$ are stated in the next observation.

\begin{observation}\label{diff}

$$
f(n)-f(n-1)=\left\{ \begin{array}{ll}
   k^2+2 &\mbox{if $n=3k$}\\
   k^2 &\mbox{if $n=3k+1$}\\
   k^2+k &\mbox{if $n=3k+2$}.\end{array}\right.
$$
\end{observation}

Two specific members of $\B(K_3)$ are $K_4^3-e$, three triples within four vertices and the tight path, with triples $abc,bcd,cde$ on vertex set $\{a,b,c,d,e\}$.

\begin{proposition}\label{btriangle} $H_n$ contains either a $K_4^3-e$ or a tight path.
\end{proposition}

\begin{proof} If the statement is not true then the trace $G=G(\{v\})$ of any vertex $v\in V(H_n)$ contains no triangles or $P_4$-s, thus $G$
must be a star forest. Therefore $f(n)-f(n-1)< \delta \le n-2$ and from Observation \ref{diff} we get contradiction for every $n\ge 6$ except for $n=7$. However,
in this case we have $f(7)+1=13$ triples in $H_7$ thus there exists a vertex of degree at least six. Applying the argument to this vertex we get a
contradiction. \end{proof}

Using Proposition \ref{btriangle}, we can select a Berge-triangle $B$ in $H_n$ such that $$E(B)=\{123,12x,23y\}$$  where $x,y\in V(H_n)\setminus [3]$. If there exists
$K_4^3-e$ in $H_n$ then we assume $x=y$, otherwise $x\ne y$.

Consider $G=G(B)=G(\{1,2,3\})$, the trace  of $H_n$ on $V(H_n)\setminus [3]$. By our assumptions, there is a loop on $x$ with label $\{1,2\}$ and a loop
 on $y$ with  label $\{2,3\}$. If $x=y$ then there is a double loop on $x$ with labels $\{1,2\},\{2,3\}$, respectively.   Set $Z=V(G)\setminus \{x,y\}$.

We partition $Z$  into eight sets $Z_I$ where $I\subseteq [3]$; a vertex $v$ belongs to $Z_I$ if the union of the labels on the edges containing $v$ is $I$.

Vertices of $Z_{\emptyset}$ are isolated in $G$.
Vertices of $Z_1,Z_2,Z_3$ cannot be incident to loops or multiple edges of $G$. Edges between $Z_{ij}$ and $Z_{ik}$ are not multiedges and labeled with
$i$ for any choice of three different indices.

\begin{proposition}\label{repr}  Assume that $z\in Z$ and $e_1,e_2,e_3$ are distinct edges of $G$ containing $z$. Then the sets $\ell(e_1),\ell(e_2),\ell(e_3)$
have no distinct representatives.
\end{proposition}
\begin{proof}
Indeed, otherwise $\{1,2,3,z\}$ would span a member of $\B$: the pairs of $\{1,2,3\}$ are covered by the three triples of the Berge-triangle $B$ and the pairs $1z,2z,3z$ can be covered by the triples containing the three $e_i$-s. This is a contradiction.
\end{proof}

We collect some consequences of Proposition \ref{repr}.

\begin{proposition}\label{mult} The following properties hold.
\begin{itemize}
\item 1.  Loops and non-loop edges in $Z$ have multiplicity at most two.
\item 2.  Double loops in $Z$ form one-vertex components in $G$ (inside $Z_{123}$).
\item 3.  Multiple edges within $Z$ are inside some $Z_{ij}$.
\item 4.  Assume $v_1v_2,v_2v_3$ are double edges in $Z_{ij}$. Then no further edges or loops can be on $v_1$ or on $v_3$.
%\item 4.
\end{itemize}
\end{proposition}

\begin{proof}Properties 1-3 are immediate consequences of Proposition \ref{repr}. By symmetry, it is enough to prove property 4 for $i=1,j=2$ and for $v_3$. Suppose $v_1v_2,v_2v_3$ are double edges  in $Z_{12}$ and $v_3z\in E(G)$ where $z\ne v_2$ and w.l.o.g. the label of $z$ is $\{1\}$.  Then we have a member of $B(K_4)$ spanned by  $\{1,2,v_2,v_3\}$ shown by the assignments
$$1v_2 \mapsto 1v_1v_2, 1v_3 \mapsto 1v_3z, 2v_2 \mapsto 2v_1v_2, 2v_3\mapsto 2v_2v_3, v_2v_3\mapsto 1v_2v_3, 12\mapsto 123,$$
leading to contradiction.   \end{proof}

Let $G^*$ be the simple graph obtained from $G$ by replacing its multiple edges by single edges and removing all loops. Obviously $|E(G^*)|=|E(G)|-s(G)$ where
$$s(G)=\sum_{v\in V(G)} \mu(v,v)+\sum_{e\in E(G)} (\mu(e)-1).$$

The {\em surplus subgraph $S$} of $G$ has vertex set $V(G)$ and its
edges are the (possible multiple) loops and multiple edges. The
connected components of $S$ are called {\em blocks}. A block $Q$  is
a {\em bad block}  if $s(S[Q])> |V(Q)|$, otherwise it is a {\em good
block}. For example, a block containing a single loop is good, a
block which has a double edge with a loop on both ends is bad.

We use the bad blocks to determine the {\em bad connected
components} of $G$, defined similarly: a connected component $C$ is
a {\em bad component}  if $s(G[C])> |V(C)|$, otherwise it is a {\em
good component}.

\begin{proposition}Every bad component contains at least one bad block.
\end{proposition}
\begin{proof} In a connected component $C$ the blocks $Q_1,\dots,Q_m$ are vertex disjoint. If all blocks are good, then
$$s(G[C])=\sum_{i=1}^m s(G[Q_i])\le \sum_{i=1}^m |V(G[Q_i])|\le |V(C)|$$ thus $C$ would be a good component, contradiction.
\end{proof}

\begin{lemma}\label{blocks} With a suitable choice of the Berge-triangle $B$, a bad component $C$ is one of the following.
\begin{itemize}
\item 1. A triple loop on $x$, implying $y=x, s(C)=3,|V(C)|=1$.
\item 2. An $m$-star on $x$, defined as follows. There is a double loop on $x$ (labeled by $\{1,2\},\{2,3\}$ consequently $y=x$).  There are $m-1\ge 2$ double edges $xz_1,xz_2,\dots,xz_{m-1}$. The set $\{z_1,\dots,z_{m-1}\}$ is independent in $G$, i.e. contains no edges of $G$. Thus $s(C)=m+1, |V(C)|=m\ge 3$.
\item 3. A double loop, $s(C)=2,|V(C)|=1$.
\item 4. A dumbbell:  a double edge with single loops on both ends or a double loop on one end, $s(C)=3, |V(C)|=2$.

\end{itemize}
\end{lemma}

\begin{proof} We prove the following statements.

\noindent
\begin{itemize}

\item {\em 1. For any choice of $B$, a bad component $C$ not containing $x$ or $y$ is a dumbbell.}

Assume $C$ is a bad component not containing $x$ or $y$. Consider any bad block $Q$ in $C$. By proposition \ref{mult} (2 and 3), $Q$  is either a double loop component in $Z_{123}$ or $Q$ is inside some $Z_{ij}$, say inside $Z_{12}$. Clearly $Q$ must contain some double edges and by Proposition \ref{mult} (4) the double edges of $Q$  form a star. If the star has one double edge then both ends must be a loop otherwise $Q$ is not bad. We claim that this is the only possibility for a bad block. Indeed, if the star has at least two double edges, then there must be at least one loop at an endpoint of the star and this is excluded in Proposition \ref{mult} (4).  Thus every bad block is a dumbbell. However the dumbbells are connected components at the same time, since any edge, say $zz_1$ with label $\{1\}$ touching the dumbbell $z_1z_2$, would define a member of $\B$ spanned by $\{1,2,z_1,z_2\}$ with the following assignments:
 $$12\mapsto 123, z_1z_2\mapsto 1z_1z_2, 1z_1\mapsto 1z_1z, 2z_1\mapsto 12z_1, 1z_2\mapsto 12z_2, 2z_2\mapsto 2z_1z_2.$$

\item {\em 2. Assume there exists a $K_4^3$ in $H_n$. Then, for some choice of $B\subset K_4^3$ and $x\in V(B)$, the bad component containing $x$ is a triple loop.}

Assume that $x$ contains a triple loop, i.e. $\{1,2,3,x\}$ spans $K_4^3$ (and $x=y$). If the triple loop is not a component of $G$ then there exists an edge $xz\in E(G)$ say with label $\{1\}$.

We claim that no vertex of $Z$ has a loop.  Indeed, suppose that $z_1\in Z$ has a loop ($z_1=z$ is possible). By symmetry, we may assume that the label of the loop is either $\{1,2\}$ or $\{2,3\}$. In both cases $\{1,2,3,x\}$ would span a member of $\B$ with the following assignments.
$$1x \mapsto 1xz, 12\mapsto 12z_1, 13\mapsto 123, 23\mapsto 23x, 2x\mapsto 12x, 3x\mapsto 13x,$$
$$1x \mapsto 1xz, 12\mapsto 12x, 13\mapsto 123, 23\mapsto 23z_1, 2x\mapsto 23x, 3x\mapsto 13x,$$
proving the claim.

Consider $G'=G(1,2,x)$, now vertex $3$ plays the role of $x$.  If the triple loop on $3$ is not a component of $G'$, there exists $3z'\in E(G')$ ($z=z'$ is possible).  The label of this edge can be one of $\{x,1,2\}$. For all the three choices $\{1,2,3,x\}$ spans  a member of $\B$ with the following assignments (the first assignment shows the choices).
\begin{equation}\label{change}
3x\mapsto 3xz', 2x\mapsto 23x, 1x\mapsto 1xz, 23\mapsto 123, 13\mapsto 13x, 12\mapsto 12x;
\end{equation}
$$13\mapsto 13z', 2x\mapsto 23x, 1x\mapsto 1xz, 23\mapsto 123, 3x\mapsto 13x, 12\mapsto 12x;$$
$$23\mapsto 23z', 2x\mapsto 23x, 1x\mapsto 1xz, 3x\mapsto 13x, 13\mapsto 123, 12\mapsto 12x,$$

leading to contradiction. Thus in $G$ or in $G'$ the triple loop is the (bad) component containing $x$.

\item {\em 3. Assume that $H_n$ contains $K_4^3-e$ but does not contain $K_4^3$. Then, for some choice of $B\subset K_4^3-e$ and $x\in V(B)$, the bad component containing $x$ is a double loop or a dumbbell or an $m$-star.}

    Starting from a $K_4^3-e$, we have a double loop on $x$, with labels $\{1,2\},\{2,3\}$ (again, $x=y$ in this case).  We determine first the block $Q$ containing $x$. If no multiedge is incident to $x$ then $Q$ is the double loop on $x$. Suppose we have a double edge $xz$ with label $\{i,j\}$ in $E(G)$.

    We claim that no other loops or edges can be incident to $z$. Indeed, a loop on $z$ must be labeled with $\{i,j\}$ and the label of an edge on $zz_1$ must be $\{i\}$ or $\{j\}$, say $\{i\}$,  because of Proposition \ref{repr}. Therefore $\{i,j,x,z\}$ would span a member of $\B$. Indeed,
$iz\mapsto ijz$ if we have a loop on $z$, $iz\mapsto izz_1$ if we have an edge $zz_1$. Then $jz\mapsto jxz, xz\mapsto ixz$ and the triple $\{ix,jx,ij\}$ has obviously a bijection to $123,12x,23x$. Thus we get a contradiction.

Therefore $Q$ consists of a double loop and some (at least one) double edges on $x$. If $Q$ is a component of $G$ we have a star or a dumbbell as a bad component. Otherwise there exists $z\in Z$ different from all vertices of the star, such that $xz\in E(G)$, w.l.o.g. with label $\{1\}$.

As in the previous case, we consider $G'=G(1,2,x)$, now vertex $3$ plays the role of $x$. With the argument of the previous paragraph we get that a star on $3$ is the bad block $Q$ of $G'$. If $Q$ is a component of $G'$ we have the star or a dumbbell as a bad component. Otherwise there exists $z'\in Z$ different from all vertices of the star such that $3z'\in E(G')$, implying $3z'w\in E(H_n)$ where $w\in \{x,1,2\}$. However, for all choices of $w$ we get a contradiction exactly as in (\ref{change}).

\item {\em 4. Assume that $H_n$ does not contain $K_4^3-e$. Then a bad component on $x$ (or on $y$) is a dumbbell on $xy$}.

Suppose there is a double edge $zz'$ in some $Z_{ij}$, say in $Z_{12}$ which is incident to a loop, say $zz$.  We can define $B=\{12z,1zz',2zz'\}$ and  $G''=G(\{1,2,z\})$. Now $z'$ can be in the role of $x$, it has a double loop, thus there is a $K_4^3-e$ in $H_n$, contradiction.  The same argument eliminates
a double edge from $x$ or from $y$ to a loop on $z$ for $z\in Z$. Likewise, a double edge $xy$ with label $\{1,2\}$ or label $\{2,3\}$ can be eliminated this way.

It is impossible to have a double edge from $x$ or from $y$ to a double edge in $Z$. Indeed, assume that there are double edges $xz,zz'$ with $z,z'\in Z$. Let the label of $xz$ be w.l.o.g. $\{1,i\}$, the label of $zz'$ must be also $\{1,i\}$. Then $\{1,i,x,z\}$ spans a member of $\B$ with assignment
$$1i\mapsto 123, 1x\mapsto 12x, ix\mapsto ixz, 1z\mapsto 1zz', iz\mapsto izz', xz\mapsto 1xz,$$
giving contradiction.

By a similar argument, it is also impossible that $x,y$ both send a double edge to the same vertex $z\in Z$:
$$1i\mapsto 123, 1x\mapsto 12x, ix\mapsto iyz, 1z\mapsto 1yz, iz\mapsto iyz, xz\mapsto 1xz.$$

We conclude that either $x,y$ belong to distinct blocks (both a single loop with some double edges), or $x,y$ are joined with a double edge with label $\{1,3\}$. In the former case the blocks are good blocks, otherwise  $x,y$ define a (bad) dumbbell block $Q$. We claim that in this case $Q$ is a component.

Indeed, if there is an edge from $\{x,y\}$ to $Z$, say $xz$ with $z\in Z$ and with label $\{1\}$ or label $\{2\}$ then $\{1,2,x,y\}$ spans a member of $\B$; and  if the label is $\{3\}$ then $\{1,3,x,y\}$ spans a member of $\B$:
$$1x\mapsto 1xz, 2x\mapsto 12x, 1y\mapsto 1xy, 2y\mapsto 23y, 12\mapsto 123, xy\mapsto 3xy;$$
$$2x\mapsto 2xz, 1x\mapsto 12x, 1y\mapsto 1xy, 2y\mapsto 23y, 12\mapsto 123, xy\mapsto 3xy;$$
$$3x\mapsto 3xz, 1x\mapsto 12x, 3y\mapsto 23y, 1y\mapsto 1xy, 13\mapsto 123, xy\mapsto 3xy,$$
proving the claim.
\end{itemize}

This finishes the proof of Lemma \ref{blocks}. \end{proof}

%We determine the bad components in the following proposition. A {\em dumbbell} is one of the following graphs with two vertices and four edges: a double edge labeled  with $i,j$ with $ij$-loops at both ends;  a double loop %and a double edge on $x$ (implying $y=x$);  a double loop on $x$, a simple edge $xz$ and a loop on $z$ (implying $y=x$).  Note that for dumbbell components $s(C)=3$.

Assume that the connected components of $G$ are $C_1,C_2,\dots$ and let
$I,J$ denote the index sets  of the good and bad components,
respectively.  Let $U\subset V(G)$ be the set  of vertices in $G$
uncovered by the bad components. Then we have
\begin{equation}\label{count}
s(G)=\sum_{j\in J} s(G[C_j])+\sum_{i\in I} s(G[C_i]) \le \sum_{j\in J}
s(G[C_j])+|U|.
\end{equation}

Now we estimate the number of edges in $G^*$. Using Proposition \ref{blocks}, assume  that we have the following bad components:
$p$ double loops, $q$ dumbbells and an $m$-star or a triple loop on $x$.
The bad components cover $m+p+2q$ vertices where the $m=1$ case occurs if $x$ is covered by a triple loop, $m\ge 3$ if $x$ is covered by an $m$-star (the $m=2$ case is considered as a dumbbell). The $m=0$ case is when no bad component covers $x$.  Then  $|U|=n-3-(m+p+2q)$ and $s(G)=m'+2p+3q$ where $m'=3$ if $m=1$, otherwise $m'=m+1$.
Therefore, by (\ref{count})

$$s(G)\le n-3-(m+p+2q)+(m'+2p+3q)=n-3+(\rho+p+q)$$
where $\rho=2$ if $m=1$ otherwise $\rho=1$.

Using $\delta(H_n)\ge f(n)-f(n-1)+1$ for the vertices $\{1,2,3\}$ (and subtracting one because the edge $\{1,2,3\}$ does not contribute to edges of $G$), we have
$$|E(G^*)|=|E(G)|-s(G)\ge 3\delta(H_n) -s(G) \ge 3(f(n)-f(n-1)) -(n-3+(\rho+p+q)).$$
The union of the bad blocks contain $m-1+q$ edges of $G^*$ ($m-1$ in the $m$-star, $q$ in the dumbbells) thus
\begin{equation}\label{edgecount}
|E(G^*[U])|\ge 3(f(n)-f(n-1)) -(n-3+(m-1+\rho+p+2q))=M.
\end{equation}
 Since $0\le m+p+2q\le n-3$, we can write $m+p+2q=\alpha(n-3)$ with $0\le \alpha \le 1$. Then $$|V(G^*[U])|=(1-\alpha)(n-3), M=3(f(n)-f(n-1))-((n-3)(\alpha+1)+\rho-1).$$

We use Tur\'an's theorem to finish the proof of Theorem \ref{exbk4}, by showing that $G^*[U]$ contains a $K_4$, thus $H_n$ contains a member of $\B$, leading to contradiction. We need to show that
$M\ge {|V(G^*[U])|^2\over 3}$, i.e.
(using that $\rho -1\le 1$)

\begin{proposition}\label{toomany}$9(f(n)-f(n-1)) -3(1+\alpha)(n-3)-(1-\alpha)^2(n-3)^2-3\ge 0$.
\end{proposition}
\begin{proof} Using Observation \ref{diff}, we have three very similar cases.
\begin{itemize}

\item 1. $n=3k$. Then $9(f(n)-f(n-1))=9(k^2+2)=(3k-3)^2+18k+9$ thus
$$ 18k+9+(3k-3)^2-3(1+\alpha)(n-3)-(1-\alpha)^2(n-3)^2-3=$$
$$=18k+9 -3(3k-3)(1+\alpha)+(3k-3)^2(2\alpha-\alpha^2)-3\ge$$
$$=\alpha((3k-3)^2-3(3k-3))+9k+15\ge 0$$
if $k\ge 2$ (for the last inequality we used $\alpha\ge \alpha^2$).

\item 2. $n=3k+1$. Then $9(f(n)-f(n-1))=9k^2=(3k-2)^2+12k-4$ thus

$$12k-4+(3k-2)^2 -3(1+\alpha)(n-3)-(1-\alpha)^2(n-3)^2-3=$$
$$=12k-4-3(3k-2)(1+\alpha) + (3k-2)^2(2\alpha-\alpha^2)-3\ge$$
$$=\alpha((3k-2)^2-3(3k-2))+3k-1\ge 0$$
if $k\ge 2$ (using again $\alpha\ge \alpha^2$).
%However, we can easily treat the case $k=3$. Then $n=10$ and if $\alpha=0$ then from (\ref{edgecount}) we have at least $3\times 8 -7=17$ edges in $G^*)$, one more than the Tur\'an graph has. Thus $\alpha\ne 0$ implying that $\alpha\ge {1\over 7}$ and the first term in the last line is $4$ compensating the $-2$ in the second term.

\item 3. $n=3k+2$. Then $9(f(n)-f(n-1))=9(k^2+k)=(3k-1)^2+15k-1$ thus

$$15k-1+(3k-1)^2 -3(1+\alpha)(n-3)-(1-\alpha)^2(n-3)^2-=$$
$$=15k-1-3(3k-1)(1+\alpha) + (3k-1)^2(2\alpha-\alpha^2)-3\ge$$
$$=\alpha((3k-1)^2-3(3k-1))+6k-1\ge 0$$
if $k\ge 2$ (using again $\alpha\ge \alpha^2$).
\end{itemize}
This proves Proposition \ref{toomany} and finishes the proof of Theorem \ref{exbk4}. \end{proof} \end{proof}

\noindent{\bf Proof of Corollary \ref{cor}.} One can easily check that the configurations of $\B$ whose absence was used in the proof of Theorem \ref{exbk4} had at most nine vertices. In fact, the only place where a nine-vertex configuration of $\B$ could appear was in the proof of  Proposition \ref{repr}, where the vertex set was $\{1,2,3,x,y,z\}$ and the possible three vertices of $e_i\setminus \{z\}$ for $i=1,2,3$.  In all other places we referred to configurations of at most seven vertices. \qed

\section{Launching the induction}
It is possible that the next theorem can be proved by a computer program, but we prove it by traditional ways. The first statement, ${\rm ex}_3(5,\B)=5$, was proved originally in \cite{AG} as a lemma to show that the $2$-color Ramsey number of $\B(K_4)$ in triple systems is equal to six.
\begin{theorem}\label{exbk4n=56} ${\rm ex}_3(5,\B)=5$ and ${\rm ex}_3(6,\B)=8$.
\end{theorem}

\begin{proof}
\smallskip
\noindent
 A pair of vertices in a triple system $H$ is {\em uncovered} if no triple of $H$ contains the pair. Let $W$ be the graph formed by the uncovered pairs.

\noindent
\smallskip
{\bf I. $n=5$.}
Five triples clearly cannot form a member of $\B$ thus we have to show ${\rm ex}_3(5,\B)< 6$. Assume $H_5$ is a triple system with six triples on vertex set $[5]$ without any member of $\B$. Observe that the maximum degree of $H_5$ is at least $\lceil {6\times 3\over 5}\rceil=4$.

\begin{itemize}
\item 1. Suppose that $W$ has an edge, for some $1\le i<j\le 5$, the pair $ij$ is not covered by any triple of $H$. By symmetry, let $i=1,j=2$. Then $H_5$ either contains the six triples meeting $\{1,2\}$ in one vertex or one of them, say $234$ is missing. In the first case the assignment $$e_{13}\mapsto 134,e_{14}\mapsto 145,e_{15}\mapsto 135,e_{34}\mapsto 234,e_{35}\mapsto 235,e_{45}\mapsto 245$$ defines a member of $\B$, otherwise the assignment $e_{34}\mapsto 234$ is replaced by the assignment $e_{34}\mapsto 345$ to get a member of  $\B$. In both cases we have a contradiction.

\item 2. Every pair of $[5]$ is covered by some triple of $H_5$.  Let $1$ be a vertex of maximum degree. If its degree is at least five, then the trace $G=G(\{1\})$ contains a $4$-cycle $23452$. Then $\{2,3,4,5\}$ spans a member of $\B$ with the two distinct triples covering pairs $24,35$ and the four triples on $1$, contradiction. Thus the degree of $1$ is four, and we may assume that $G$ does not contain a $4$-cycle. Thus $G$ is a triangle, say $2342$ with a pendant edge $45$. There are two triples of $H_5$ within $\{2,3,4,5\}$ and apart from symmetry there are four possible choices (with respect to the four edges of $G$). In each case we can easily find a member of $\B$ spanned by $\{1,2,3,4\}$ as shown below (the first two assignments are the ones using the two edges of $H_5$ inside $\{2,3,4,5\}$.

$$34\mapsto 234, 23\mapsto 235, 12\mapsto 123, 13\mapsto 134, 14\mapsto 145, 24\mapsto 124;$$
$$24\mapsto 234, 34\mapsto 345, 12\mapsto 124, 13\mapsto 134, 14\mapsto 145, 23\mapsto 123;$$
$$23\mapsto 235, 24\mapsto 245, 12\mapsto 124, 13\mapsto 123, 14\mapsto 145, 34\mapsto 134;$$
$$24\mapsto 245, 34\mapsto 345, 12\mapsto 124, 13\mapsto 134, 14\mapsto 145, 23\mapsto 123.$$
\end{itemize}

\noindent
\smallskip
{\bf II. $n=6$.}
Suppose we have $9=f(6)+1$ triples in $H_6$. The maximum degree is at least $\lceil {9\times 3\over 6}\rceil=5$. The minimum degree is at least $4$ otherwise deleting the vertex we get six triples on five vertices contradicting Case I.

This implies that $W$ has no vertex with degree at least two, thus its edges form a matching $M$.  If $|M|=3$ then some vertex on $M$ has degree at least $5$, contradiction.

\begin{itemize}
\item 1.  Suppose $|M|=2$, say $12,34$ are the edges of $W$ and some vertex on $M$, say $1$, has degree $5$. Then the five triples containing $1$ are determined:
$e_1=135,e_2=136,e_3=145,e_4=146,e_5=156$. Then, since the degree of $2$ is at least four, one of the triples $235,236$ must be an edge of $H_6$. Then $\{1,3,5,6\}$ spans a member of $\B$ in $H_6$, shown by one of the assignments
$$35\mapsto 235, 13\mapsto 135, 15\mapsto 145, 16\mapsto 146, 36\mapsto 136, 56\mapsto 156,$$
$$36\mapsto 236, 13\mapsto 136, 15\mapsto 145, 16\mapsto 146, 35\mapsto 135, 56\mapsto 156.$$

\item 2. Suppose $|M|=2$, say $12,34$ are the edges of $W$ and vertices $1,2,3,4$ have degree four. Then $d(5)+d(6)=11$ implying that the trace $G(\{5,6\})$ of $H_6$ must contain a $P_4$, say $1,3,2,4$ with double edges, and (by Property 4. in Proposition \ref{mult})  $\{2,3,5,6\}$ spans a member of $\B$.

\item 3.  Suppose $|M|=1$, $12$ is an edge of $W$.   Now the trace $G=G(\{1,2\})$ has no loops. Note that $G$ cannot contain a $4$-cycle, say  $3,4,5,6,3$ because then with the edges covering the pairs $35,46$ we get a member of $\B$ on $\{3,4,5,6\}$.  Since $G$ contains at least eight edges, the trace must be a triangle with a pendant edge formed by double edges, say $34,45,35,36$.  Now $\{1,3,4,5\}$ spans a member of $\B$, contradiction.

\item 4. Suppose $|M|=0$ i.e. $W$ has no edges, every pair of $[6]$ is covered by some edge of $H_6$. No vertex of $H_6$ has degree at least 6. Indeed, if $1$ is such a vertex, then the trace $G=G(\{1\})$ has at least six edges and either contains a $4$-cycle, say $23452$ or isomorphic to two edge disjoint triangles, say with edges $23,34,24,25,26,56$. In the first case $\{2,3,4,5\}$ spans a member of $\B$ with assignments
$$23\mapsto 123, 34\mapsto 134, 45\mapsto 145, 25\mapsto 125$$
and the pairs $24,35$ are mapped to the triples containing them.  In the second case one can easily see that if any triple in $\{3,4,5,6\}$ is not in $H_6$ then we have a member of $\B$. For example, assume that $345$ is not in $H_6$, thus the edges $e_{35},e_{45}$ that cover the pairs $35,45$ are different and $\{2,3,4,5\}$ spans a member of $\B$ with the assignments
$$35\mapsto e_{35}, 45\mapsto e_{45}, 23\mapsto 123, 24\mapsto 124, 25\mapsto 125,  34 \mapsto 134,$$
giving a contradiction.

Thus the degree sequence of $H_6$ is $4,4,4,5,5,5$. Counting the codegrees, i.e. the number of times the vertex pairs of $H_6$ are covered by the triples, we get that at least three pairs $x_iy_i$ have codegree one. Indeed, otherwise the sum of the codegrees is at least $2+13\times 2>3\times 9$, contradiction.

 We claim that for $i=1,2,3$, the degrees of $x_i,y_i$ are equal to $4$. Indeed, otherwise the trace $G=G(\{x_i,y_i\})$ has at least seven edges, implying that it must contain a pair of incident double edges plus a further loop or edge incident to an endpoint, leading to a member of $\B$. In fact, this is Property 4. in Proposition \ref{mult}.

We are left with one case to consider: the pairs $x_iy_i$ form a triangle, say $1,2,3,1$ their degrees are equal to $4$ and vertices $4,5,6$ has degree $5$. Then $\{1,2,3\}$ cannot be in $E(H_6)$, that would allow at most eight triples. The trace $G=G(\{1,2,3\})$ on $\{4,5,6\}$ cannot contain a triple edge, say $45$ with label $\{1,2,3\}$ because then $\{1,2,3,4\}$ would span a member of $\B$. Indeed,
$$14\mapsto 145, 24\mapsto 245, 34\mapsto 345$$
and the pairs $12,13,23$ are covered by the triples of $H_6$ containing them. This implies that the double edges $45,56,46$ are labeled w.l.o.g. with $\{1,2\},\{1,3\},\{2,3\}$, respectively, and each vertex of $T$ has one loop.  Now $\{1,2,3,4\}$ spans a member of $\B$ again,
$$14\mapsto 145, 24\mapsto 245, 34\mapsto 346$$
and the pairs $12,13,23$ are covered by the edges of $H_6$ containing them. This is a contradiction, finishing the proof of Theorem \ref{exbk4n=56}.
\end{itemize}
\end{proof}

\noindent{\bf Acknowledgement.} Thanks to Maria Axenovich, D\'aniel Gerbner and Zolt\'an F\"uredi for conversations on the subject.

\eject

\end{document}